\documentclass[a4paper,12pt,twoside]{article}
\date{09/06/2021}

\usepackage[english]{babel}
\usepackage{amsfonts}
\usepackage{lmodern,fancyhdr,lastpage,nccfoots,caption}
\usepackage{hyperref}
\hypersetup{
    pdfstartview={FitH},     									
    pdftitle={Stratifications on the Nilpotent Cone of the Hitchin map},             				
    pdfauthor={P.\,B.~Gothen~and~R.\,A.~Z\'u\~niga-Rojas},     						
    pdfsubject={Vector bundles on curves and their moduli},           				
    pdfkeywords={Higgs Bundles, Hitchin Pairs, Hodge Bundles, Moduli Spaces, Morse Theory, Nilpotent Cone, Vector Bundles.},
    pdfnewwindow=true,        									
    colorlinks=true,          									
    linkcolor=blue,           									
    citecolor=magenta,            									
    urlcolor=blue             									
}
\usepackage[utf8]{inputenc}
\usepackage{graphicx,xcolor}
\usepackage{amsmath,amssymb,amsthm}
\usepackage{enumerate}
\usepackage[a4paper,margin=1in]{geometry}
\usepackage[all]{xy}

\theoremstyle{definition} 
\newtheorem{Def}{Definition}[section]

\theoremstyle{plain} 
\newtheorem{Th}[Def]{Theorem}
\newtheorem*{nonum-Th}{Theorem}     

\newtheorem*{nonum-Cor}{Corollary}  

\theoremstyle{remark} 
\newtheorem{Rmk}[Def]{Remark}

\numberwithin{equation}{section}

\newcommand{\sse}{\subseteq}	      
\newcommand{\suchthat}{\;\;|\;\;}

\renewcommand{\leq}{\leqslant}
\renewcommand{\geq}{\geqslant}

\renewcommand{\phi}{\varphi}

\newcommand{\bC}{\mathbb{C}}        
\newcommand{\bN}{\mathbb{N}}        
\newcommand{\bR}{\mathbb{R}}        

\newcommand{\sM}{\mathcal{M}}       

\newcommand{\dbar}{\bar{\partial}}

\DeclareMathOperator{\tr}{tr}
\DeclareMathOperator{\rk}{rk}
\DeclareMathOperator{\im}{im}

\DeclareMathOperator{\Hom}{Hom}
\DeclareMathOperator{\End}{End}

\newcommand{\word}[1]{\ \text{#1}\ } 

\newcounter{a}
\ifodd\thea\else\stepcounter{a}\fi

\fancypagestyle{plain}{%
\fancyhf{}

}
\defineshorthand{"-}{\nobreak\hskip0pt\discretionary{-}{-}{-}\nobreak\hskip0pt}

\newcommand*{\QEDA}{\hfill\ensuremath{\boxminus}} 
\DeclareRobustCommand{\QEDA}{\ifmmode
  \else \leavevmode\unskip\penalty9999 \hbox{}\nobreak\hfill \fi
  \quad\hbox{\QEDAasymbol}}
\newcommand{\QEDAasymbol}{$\boxminus$} 

\begin{document}
\thispagestyle{plain}

\begin{center}
\Large
\textsc{Stratifications on the Nilpotent Cone\\of the moduli space of Hitchin pairs}
\end{center}

\begin{center}
  08/06/2021
\end{center}

\begin{center}
  \textit{Peter B. Gothen\footnote{Partially supported by Centro de Matem\'atica da Universidade do Porto (CMUP),
      financed by national funds through FCT---Funda\c{c}{\~a}o para a
      Ci{\^e}ncia e a Tecnologia, I.P., under the project UIDB/00144/2020.}}\\
  \small Centro de Matem\'atica da Universidade do Porto and\\
  \small Departamento de Matem\'atica\\
  \small Faculdade de Ci\^encias da Universidade do Porto\\
  \small Rua do Campo Alegre, s/n\\
  \small 4197-007 Porto, Portugal\\
  \small e-mail: \texttt{pbgothen@fc.up.pt}
\end{center}

\begin{center}
  \textit{Ronald A. Z\'u\~niga-Rojas\footnote{Supported by 
  Universidad de Costa Rica through Escuela de Matem\'atica, specifically through CIMM and CIMPA, under projects {\tt 820-B5-202}, {\tt 820-B8-224} and {\tt 821-C1-010}. Partially supported by FCT (Portugal) through grant {\tt SFRH/BD/51174/2010}.}}\\
  \small Centro de Investigaciones Matem\'aticas y Metamatem\'aticas CIMM y\\
  \small Centro de Investigaci\'on en Matem\'atica Pura y Aplicada CIMPA\\
  \small Escuela de Matem\'atica, Universidad de Costa Rica UCR\\
  \small San Jos\'e 11501, Costa Rica\\
  \small e-mail: \texttt{ronald.zunigarojas@ucr.ac.cr}
\end{center}

\noindent
\textbf{Abstract.}  

We consider the problem of finding the limit at infinity
(corresponding to the downward Morse flow) of a Higgs bundle in the
nilpotent cone under the natural $\bC^*$-action on the moduli
space. For general rank we provide an answer for Higgs bundles with
regular nilpotent Higgs field, while in rank three we give the complete
answer. Our results show that the limit can be described in terms of
data defined by the Higgs field, via a filtration of the underlying
vector bundle.

\medskip


\noindent\textbf{MSC~2020 classification:} Primary \texttt{14H60}; Secondary \texttt{14D07}.	

\section*{Introduction}
\label{sec:0:intro}
\addcontentsline{toc}{section}{Introduction}

Over thirty three years ago, Hitchin~\cite{hit2} introduced Higgs
bundles on Riemann surfaces through dimensional reduction of the
self-duality equations from $\bR^4$ to $\bR^2$, and they appeared in
the work of Simpson \cite{sim1} motivated by uniformisation problems
for higher dimensional varieties. Since then, the moduli
space of Higgs bundles has become an important topic of research in
many areas of geometry and mathematical physics and there are even
ramifications to number theory via the Langlands programme. Much more
detailed information and many references to relevant work can be found
in the following selection
of (mainly) expository papers: \cite{bgg}, \cite{garcia}, \cite{got},
\cite{hausel:2013}, \cite{rayan}, \cite{schaposnik}. 

A Higgs bundle on a Riemann surface is a pair consisting of a
holomorphic vector bundle together with an endomorphism valued
holomorphic one-form, called the Higgs field.

Taking the characteristic polynomial of the Higgs field defines the Hitchin
map, which is a proper map from the moduli space of Higgs bundles to a
vector space. It makes the moduli space of Higgs
bundles into an algebraic completely integrable Hamiltonian
system, and thus the generic fibre of the Hitchin map is an abelian
variety. On the other hand, the fibre over zero, named the nilpotent
cone by Laumon~\cite{laumon}, is highly singular and it encodes many important
properties of the moduli space: for example, the moduli space deformation
retracts onto it.

Another important attribute of the moduli space of Higgs bundles is
that it carries an action of the non-zero complex numbers $\bC^*$ via
multiplication on the Higgs field. The limit of the action on a Higgs
bundle of $z\in\bC^*$
as $z\to 0$ always exists, and thus the moduli space has an associated
Bia{\l}ynicki-Birula stratification.
On the other hand, the limit as $z\to\infty$  exists if and
only if the Higgs bundle belongs to the nilpotent cone. These limits
are fixed points of $\bC^*$-action. Such fixed points are
known as Hodge bundles and are all contained in the nilpotent cone. 

In our earlier work \cite{gzr,z-r1} (see also \cite{z-r0,z-r2}) we investigated the limit as $z\to 0$ of any Higgs bundle and its relation to the Harder--Narasimhan filtration of the underlying vector bundle, in order to better understand the relation between the Bia{\l}ynicki-Birula and Shatz stratifications of the moduli space (the latter being defined by the Harder--Narasimhan type). The case of rank two had already been considered by Hitchin \cite{hit2}, who observed that in this case the two stratifications coincide. This is no longer the case in higher rank and, indeed, the general problem is quite intricate; a complete solution is given in \cite{gzr} for rank 3.  The companion problem of finding the limit of Higgs bundle in the nilpotent cone as $t\to\infty$ was also considered in the second author's PhD thesis~\cite{z-r2} and the result of the present article are essentially contained there. We have decided to write them up here in view of recent interest in the fine structure of the Bia{\l}ynicki-Birula stratification of the nilpotent cone.

Our main results are as follows. In the case when the Higgs field of a
Higgs bundle in the nilpotent cone is a regular nilpotent, there is an
associated graded Higgs bundle induced from the filtration obtained by
taking the kernels of iterates of $\Phi$. This Higgs bundle is in fact
a Hodge bundle and we show that it is exactly the limit of the action
of $z\in\bC^*$ on the original Higgs bundle as
$z\to\infty$. The precise statement is in
Theorem~\ref{NilpotentConeRkR} below. On the other hand, when the
Higgs field is not a regular nilpotent, the situation is again more
intricate. We analyse the situation completely in the case of rank $3$
and show that there is a refinement of the aforementioned filtration
obtained using also the image of the Higgs field, which allows to
identify the limit as a function of topological invariants of the filtration. It is
notable that the answer depends only on properties of the Higgs field
and not on the stability properties of the underlying vector bundle
(as opposed to situation for $z\to 0$).  The precise statement is in
Theorem~\ref{NilpotentConeRk3} below. 

We mention that in this paper we work with the moduli space of
Hitchin pairs, since our results and methods are in this
generality: this means that we allow the Higgs field to be twisted by
any holomorphic line bundle of degree greater than or equal to that of
the canonical bundle of the Riemann surface, rather than just the
canonical bundle.

This paper is organised as follows. In Section~\ref{sec:1} we give
some necessary preliminaries about Hitchin pairs, Higgs bundles and
their moduli spaces, and we introduce the Hitchin map, the Nilpotent
Cone and the $\bC^*$-action. Then, in Section~\ref{sec:regular}, we
present the result in general rank for Hitchin pairs with regular
nilpotent Higgs field. Finally, in Section~\ref{sec:3}, we give the
complete result for Hitchin pairs of rank 3 with nilpotent Higgs
field.

\section{Preliminaries on Hitchin pairs and their moduli}
\label{sec:1}

In this section we review some standard facts about Hitchin pairs and
their moduli. Details can be found in, for example, Hitchin
\cite{hit2,hitchin:1987b}, Nitsure
\cite{nit} and 
Simpson \cite{sim2}. 

Let $X$ be a compact, connected and oriented Riemann surface of genus
$g \geq 2$ and let $L\to X$ be a holomorphic line bundle.

\begin{Def}
  A {\em Hitchin pair} over $X$ is a pair $(E, \Phi)$ where the
  {\em underlying vector bundle} $E \to X$ is a holomorphic vector
  bundle and the {\em Higgs field} $\Phi: E \to E\otimes L$ is
  holomorphic.
\end{Def}

If we need to specify the line bundle $L$, we say that the Hitchin
pair $(E,\Phi)$ is \emph{twisted by $L$}.

\begin{Def}
  A {\em Higgs bundle} over $X$ is a Hitchin pair $(E, \Phi)$ twisted
  by the canonical line bundle $K=K_X=T^{*}X$.
\end{Def}

The {\em slope} of a vector bundle $E$ is the quotient between its
degree and its rank: 
\[
 \mu(E)
 =
 \deg(E)/\rk(E). 
\]
Recall that a vector bundle $E$ is {\em semistable} if
$\mu(F)\leq\mu(E)$ for all non-zero holomorphic 
subbundles $F \subseteq E$, {\em stable} if it is semistable and strict
inequality holds for all non-zero proper $F$, and {\em polystable} if
it is the direct sum of stable bundles, all of the same slope. The
\emph{slope of a Hitchin pair} is the slope of its underlying vector
bundle and the stability condition is defined analogously to the
vector bundle situation, except that the slope condition is applied only
to $\Phi$-invariant subbundles, i.e., holomorphic subbundles $F\subseteq E$ such that
$\Phi(F)\subseteq F\otimes L$.

The moduli space $\sM_{L}(r,d)$ of $S$-equivalence classes of
semistable rank $r$ and degree $d$ Higgs bundles was first constructed by
Nitsure \cite{nit}. The points of $\sM_{L}(r,d)$ correspond to
isomorphism classes of polystable Hitchin pairs.  When $r$ and $d$ are
co-prime any semistable Hitchin pair is automatically stable. Henceforth we shall assume that we are in
this situation and that $\deg(L)\geq 2g-2$. Then $\sM_{L}(r,d)$ is a
smooth complex manifold of complex dimension 
\[
  r^2\deg(L)+1+\dim H^1(X,L). 
\]
The moduli space is non-compact but there is a proper map, the so-called \emph{Hitchin 
map}, defined by:
\begin{equation}
 \begin{array}{r c l}
   \chi: \sM_{L}(r,d)   & \longrightarrow & H^0(X,L) \oplus {\dots} \oplus H^0(X,L^r)\\ 
   \big[ (E, \Phi) \big] & \longmapsto     & \big(\tr(\Phi),{\dots},\det(\Phi)\big)
 \end{array}
  \label{(r,d)LHitchinMap}
\end{equation}
whose components are holomorphic sections obtained as the coefficients of the (fibrewise) characteristic
polynomial of $\Phi$.
When $L = K$, the moduli space is a holomorphic symplectic manifold and the Hitchin map endows it with an algebraically completely integrable Hamiltonian system whose generic fibre is an abelian variety. (For general $L$, this has been generalised to the Poisson setting by Bottacin \cite{bottacin:1995} and Markman \cite{markman:1994}.) On the other hand, the fibre of the Hitchin map over zero, 
\[
\chi^{-1}(0):=
\left\{ 
\big[(E,\Phi)\big] \in \sM_L(r,d)\suchthat \chi(\Phi) = 0 
\right\} 
\]
is known as the \emph{Nilpotent Cone} in the moduli space, and has a complicated structure with several irreducible components. 

Next we review some standard facts about the holomorphic action of the
multiplicative group $\bC^{*}$ on
$\sM_{L}(r,d)$. The action is defined by the multiplication: 
\[
z\cdot (E,\Phi)\mapsto (E,z\cdot \Phi). 
\]
The limit $(E_0,\phi_0)=\displaystyle \lim_{z\to0} (E, z\cdot\Phi)$ 
exists for all $(E,\Phi)\in \sM(r,d)$.
On the other hand, it follows from the properties of the Hitchin map
that the limit $(E^\infty,\Phi^\infty)=\displaystyle \lim_{z\to\infty}(E, z\cdot\Phi)$ 
exists if and only if $(E,\Phi)$ belongs to the nilpotent cone $\chi^{-1}(0)$.
When the limit of $(E, z\cdot\Phi)$ as $z\to 0$ or $z\to\infty$ exists
it is fixed by the
$\bC^{*}$-action. Moreover, a Hitchin pair $(E,\Phi)$ is a fixed point of the $\bC^*$-action 
if and only if it is a \emph{Hodge bundle}, i.e., there is a decomposition
\begin{math}
  E=\bigoplus_{j=1}^p E_j
\end{math}
with respect to which the Higgs field has weight one: $\Phi\colon
E_j\to E_{j+1}\otimes L$. The \emph{type} of the Hodge bundle
$(E,\Phi)$ is $\big(\rk(E_1),\dots,\rk(E_p)\big)$.

We shall consider the moduli space from the complex analytic
point of view. For this, fix a $C^\infty$ complex vector bundle
$\mathcal{E}\to X$ of rank $r$ and degree $d$. A holomorphic
structure on $\mathcal{E}$ is given by a $\dbar$-operator
\[
  \dbar_E\colon A^0(\mathcal{E}) \to {A}^{0,1}(\mathcal{E})
\]
and we thus obtain a holomorphic vector bundle 
$E=(\mathcal{E},\dbar_E)$.  A Hitchin pair 
$(E,\Phi)$ arises from a pair $(\dbar_E,\Phi)$ consisting of a
$\dbar$-operator and a Higgs field $\Phi\in
A^{0}\bigl(\End(E)\otimes L\bigr)$ which is holomorphic, i.e.,
$\dbar_{E,L}\Phi=0$, where $\dbar_{E,L}$ denotes the $\dbar$-operator
on the underlying smooth bundle of $\End(E)\otimes L$ defining the
holomorphic structure. The natural
symmetry group is the {\em complex gauge group}
\[
  \mathcal{G}^\bC = 
  \left\{
  g\colon \mathcal{E}\to \mathcal{E}\suchthat 
  g \word{is a} C^{\infty}-\textmd{bundle isomorphism}
  \right\},
\]
which acts on pairs $(\dbar_E,\Phi)$ in the standard way:
\begin{displaymath}
  g\cdot(\dbar_E,\Phi)=(g\circ\dbar_E\circ g^{-1},g\circ\Phi\circ g^ {-1}).
\end{displaymath}
The moduli space can then be viewed as the quotient\footnote{See {Atiyah~\&~Bott~\cite[Section~14]{atbo}} for general holomorphic bundles, Hitchin~\cite[Section~3]{hit2} for Higgs bundles and also {Hausel~\&~Thaddeus~\cite[Section~8]{hath1}} for Hitchin pairs.}
\begin{displaymath}
  \sM_{L}(r,d) = 
  \left\{
  (\dbar_E,\Phi)\suchthat 
  \Phi \word{is holomorphic and}
  (E,\Phi)\ \textmd{is polystable}
  \right\} 
  / \mathcal{G}^\bC.
\end{displaymath}

\section[Regular nilpotents]{Limit at infinity for regular nilpotent Higgs field}
\label{sec:regular}

Let $(E,\Phi)$ be a stable Hitchin pair of rank $r$ and degree $d$ coprime: $GCD(r,d) = 1$, which represents a point in the nilpotent cone $\chi^{-1}(0) \sse \mathcal{M}_{L}(r,d)$. Let $p \in \bN$ be the least positive integer such that $\Phi^p=0$ and $\Phi^{p-1} \neq 0$. Then $p\leq r$ and $\Phi$ is \emph{regular} if $p=r$. Since we are working over a Riemann surface, taking the saturation of the kernel sheaf of $\Phi^{p-j+1}\colon E \to E\otimes L^{p-j+1}$ defines a subbundle $E_j\subset E$. We obtain in this way a filtration of $E$,
\begin{equation}
  \label{eq:1}
  E = E_{1} \supset E_{2} \supset \dots \supset E_{r} \supset E_{r+1} = 0
\end{equation}
and, clearly,
\begin{equation}
  \label{eq:2}
  \Phi(E_j)\subseteq E_{j+1}\otimes L.
\end{equation}
Define $\bar{E}_j = E_j / E_{j+1}$. Then, in view of \eqref{eq:2},
$\Phi$ induces a map
$\varphi_j\colon \bar{E}_j\to \bar{E}_{j+1}\otimes L$. Note that if $\Phi$ is
regular then the inclusions in \eqref{eq:1} are all strict of co-dimension one.
Thus, when $\Phi$ is regular, we obtain
a Hodge bundle of rank $r$ and degree $d$ of type $(1,\dots,1)$:
\begin{equation}
  \label{eq:4}
  (\bar{E},\bar{\Phi}) =
  \biggl(\bigoplus_{j=1}^r\bar{E}_j,\sum_{j=1}^{r-1}\varphi_j\biggr)
= 
 \biggl(\bigoplus_{j=1}^r \bar{E}_j,
 \left(
    \begin{array}[c]{c c c c c}
          0       &     \dots     &   \dots   &     \dots     &    0\\
     \varphi_{1}  &      0      &   \dots   &     \dots     &    0\\
          0       & \varphi_{2} &    0    &     \dots     &    0\\
       \vdots     &   \ddots    & \ddots  &    \ddots   & \vdots\\
          0       &     \dots     &    0    & \varphi_{r-1} &    0 
    \end{array}
 \right)
 \biggr).
 \end{equation}

\begin{Th}\label{NilpotentConeRkR} 
Let $(E,\Phi)$ be a stable Hitchin pair of rank $r$ and degree $d$ coprime: $GCD(r,d) = 1$, which represents a point in the nilpotent cone $\chi^{-1}(0) \sse \mathcal{M}_{L}(r,d)$ and assume that $\Phi$ is a regular nilpotent, i.e., $\Phi^{r-1}\neq 0$. Then $\displaystyle \lim_{z\to\infty}(E,z\cdot\Phi)=(\bar{E},\bar{\Phi})$, where $(\bar{E},\bar{\Phi})$ is given by \eqref{eq:4}. In particular the limit is a Hodge bundle of type $(1,\dots,1)$.
\end{Th}

\begin{proof}
Using the notation introduced above we may consider a smooth splitting
\[
 E \underset{C^{\infty}}{\cong} \bigoplus_{j=1}^r \bar{E}_j.
\]
Then the Higgs field takes the triangular form:
\[
  \Phi =
 \left(
    \begin{array}[c]{c c c c c}
          0         &    \dots     &      \dots       &      \dots       &    0\\
     \varphi_{21}   &      0       &      \dots       &      \dots       &    0\\
     \varphi_{31}   & \varphi_{32} &       0          &      \dots       &    0\\
       \vdots       &    \ddots    &     \ddots       &     \ddots       & \vdots\\
     \varphi_{r, 1} &    \dots     & \varphi_{r, r-2} & \varphi_{r, r-1} &    0 
    \end{array}
 \right)
\]
where $\varphi_{ij}\colon\bar{E}_{j} \rightarrow \bar{E}_{i} \otimes L$ and we note that $\varphi_{j,j-1}=\varphi_j$ in the notation introduced above. The $\dbar$-operator defining the holomorphic structure on $E$ is of the form:
\[
  \dbar_E =
 \left(
    \begin{array}[c]{c c c c }
     \dbar_1      &    0     & \dots          &    0  \\
     \beta_{21}   & \dbar_2  & \ddots         & \vdots\\
     \vdots       & \ddots   & \ddots         &    0\\
     \beta_{r,1}  & \dots    & \beta_{r,r-1}  & \dbar_r
    \end{array}
 \right)
\]
where $\dbar_j$ is the corresponding holomorphic structure of $\bar{E}_j$, and $\beta_{ij} \in \Omega^{0,1}\big(X,\Hom(\bar{E}_j,\bar{E}_i)\big)$.

We now define a family of complex $C^\infty$-gauge transformations $g(z) \in \mathcal{G}^{\bC}$ by:
\[
  g(z)=
 \left(
    \begin{array}[c]{c c c c }
        1   &    0   & \dots  &    0\\
        0   &    z   & \ddots & \vdots\\
     \vdots & \ddots & \ddots &    0\\
        0   & \dots  &    0   & z^{r-1}
    \end{array}
 \right).
\]
Then
\begin{multline*}
  g^{-1}(z)(z \cdot \Phi)g(z) = 
  \\ \\
  \left(
    \begin{array}[c]{c c c c }
      1      & 0      & \dots  & 0\\
      0      & z^{-1} & \ddots & \vdots\\
      \vdots & \ddots & \ddots & 0\\
      0      & \dots  & 0      & z^{1-r}
    \end{array}
  \right) \left(
    \begin{array}[c]{c c c c}
      0          &      \dots  &       \dots       &    0\\
      z\varphi_{21}    &       0     &       \dots       &    0\\
      \vdots        &      \ddots &      \ddots       & \vdots\\
      z\varphi_{r, 1}  &      \dots  & z\varphi_{r, r-1} &    0 
    \end{array}
  \right) \left(
    \begin{array}[c]{c c c c }
      1      & 0      & \dots  & 0\\
      0      & z      & \ddots & \vdots\\
      \vdots & \ddots & \ddots & 0\\
      0      & \dots  & 0      & z^{r-1}
    \end{array}
  \right) =
  \\ \\
\left(
    \begin{array}[c]{c c c c c}
      0           &      \dots    &         \dots          &      \dots       &    0\\
      \varphi_{21}       &       0       &         \dots          &      \dots       &    0\\
      z^{-1}\varphi_{31}    &  \varphi_{32} &          0             &      \dots       &    0\\
      \vdots          &     \ddots    &        \ddots          &      \ddots      & \vdots\\
      z^{1-p}\varphi_{r, 1} &      \dots    & z^{-1}\varphi_{r, r-2} & \varphi_{r, r-1} &    0 
    \end{array}
  \right) \xrightarrow[z\rightarrow \infty]{} \left(
    \begin{array}[c]{c c c c c}
      0         &      \dots     &  \dots   &   \dots          &    0\\
      \varphi_{21}     &       0        &  \dots   &   \dots          &    0\\
      0         &  \varphi_{32}  &    0     &   \dots          &    0\\
      \vdots        &     \ddots     &  \ddots  &   \ddots         & \vdots\\
      0         &      \dots     &    0     & \varphi_{r, r-1} &    0 
    \end{array}
  \right) =:\Phi^\infty,
\end{multline*}
and also
\begin{multline*}
  g^{-1}(z)\ \dbar_E\ g(z) = 
  \\ \\
  \left(
    \begin{array}[c]{c c c c }
      1      & 0      & \dots  & 0\\
      0      & z^{-1} & \ddots & \vdots\\
      \vdots & \ddots & \ddots & 0\\
      0      & \dots  & 0      & z^{1-r}
    \end{array}
  \right) \left(
    \begin{array}[c]{c c c c }
      \dbar_1    &    0     &     \dots      &    0  \\
      \beta_{21} & \dbar_2  &    \ddots      & \vdots\\
      \vdots     & \ddots   &    \ddots      &    0\\
      \beta_{r, 1} & \dots    & \beta_{r, r-1} & \dbar_r
    \end{array}
  \right) \left(
    \begin{array}[c]{c c c c }
      1      & 0      & \dots  & 0\\
      0      & z      & \ddots & \vdots\\
      \vdots & \ddots & \ddots & 0\\
      0      & \dots  & 0      & z^{r-1}
    \end{array}
  \right) = 
  \\ \\
  \left(
    \begin{array}[c]{c c c c }
      \dbar_1         &       0     &         \dots        &       0  \\
      z^{-1}\beta_{21}    &   \dbar_2   &       \ddots         &   \vdots\\
      \vdots          &   \ddots    &       \ddots         &       0\\
      z^{-r}\beta_{r, 1}  &   \dots     & z^{-1}\beta_{r, r-1} &   \dbar_r
    \end{array}
  \right) \xrightarrow[z\rightarrow \infty]{} \left(
    \begin{array}[c]{c c c c }
      \dbar_1 &    0    &     \dots    &     0  \\
      0     & \dbar_2 &    \ddots    &  \vdots\\
      \vdots  & \ddots  &    \ddots    &     0\\
      0     & \dots   &       0      &  \dbar_r
    \end{array}
  \right) =:\dbar_E^\infty,
\end{multline*}
where the limits are taken in the configuration space of all pairs $(\dbar_E,\Phi)$, up to gauge equivalence. Moreover, the fact that $\dbar_E\Phi=0$ immediately implies that $\dbar_E^\infty\Phi^\infty=0$ and, clearly, the Hitchin pair defined by $(\dbar_E^\infty,\Phi^\infty)$ is $(\bar{E},\bar{\Phi})$. Hence, in order to prove that the stated limit is valid in the moduli space, it only remains to prove that this Hitchin pair is stable. For this we observe that the only $\bar{\Phi}$-invariant subbundles of $\bar{E}$
are those of the form
\begin{displaymath}
  \bar{E}_l\oplus\bar{E}_{l+1}\oplus\dots\oplus\bar{E}_{r}\subseteq\bar{E}
\end{displaymath}
and note that the slope of such a subbundle equals that of $E_l\subseteq E$ because they are isomorphic as $C^\infty$-bundles. Thus, since the subbundle $E_l\subseteq E$ is $\Phi$-invariant, the stability of $(\bar{E},\bar{\Phi})$ follows from that of $(E,\Phi)$.
\end{proof}

\begin{Rmk}
  Since in rank two a nilpotent Higgs field is either zero or regular, the preceding theorem, together with the results of our previous paper \cite{gzr}, gives a complete description of the closure of the $\bC^*$-orbit of a rank 2 Hitchin pair in the nilpotent cone. Indeed, as we have just seen, the type of the limiting Hodge bundle as $z\rightarrow\infty$ is determined by the Higgs field and, from \cite[Corollary~3.2]{gzr}, the type of the limiting Hodge bundle as underlying vector bundle. These observations were already made by Hausel~\cite{hau}.
\end{Rmk}

\section[Rank Three Hitchin Pairs]{Rank Three Hitchin Pairs in the Nilpotent Cone}
\label{sec:3}

In this section we determine the limit $\displaystyle\lim_{z\rightarrow
  \infty}(E,z\cdot \Phi)$ for any rank 3 Hitchin pair $(E,\Phi)$ in
the nilpotent cone $\chi^{-1}(0)\subset (E,\Phi) \in
\sM_L(3,d)$. Since the case $\Phi=0$ is trivial and the case when
$\Phi$ is a regular nilpotent has already been covered, it only
remains to consider the case when $\Phi\neq 0$ and $\Phi^2 \equiv
0$. For completeness we state the full result.

\begin{Th}\label{NilpotentConeRk3}
Let $(E,\Phi)$ be a stable Hitchin pair of rank $3$ and degree $d$ which
represents a point in the nilpotent cone
$\chi^{-1}(0) \sse \mathcal{M}_{L}(3,d)$. Then one of the following alternatives holds: 
\begin{enumerate}[(a)]
\item The Higgs field $\Phi$ vanishes identically and $\displaystyle
  \lim_{z\rightarrow \infty}(E,z\cdot \Phi) = (E,\Phi)  = (E,0)$.
\item The Higgs field $\Phi$ is a regular nilpotent (i.e.,
  $\Phi^{2}\neq 0$) and there is a filtration
 \[
  E = E_1 \supset E_2 \supset E_3 \supset E_{4} = 0
 \]
with each step of co-dimension one and such that
$\Phi(E_j) \subset E_{j+1} \otimes L$ for $j=1,2,3$.
In this case, 
\begin{equation}
  (E^{\infty},\Phi^{\infty}) = \displaystyle \lim_{z\rightarrow \infty}(E,z\cdot \Phi) = 
 \Big(\bar{E}_1 \oplus \bar{E}_2 \oplus \bar{E}_3,
 \left(
    \begin{array}[c]{c c c}
           0       &      0       &    0\\
     \varphi_{1}  &      0       &    0\\
           0       & \varphi_{2} &    0
    \end{array}
 \right)
 \Big)
 \end{equation}
is a Hodge bundle of type $(1,1,1)$ where
\[
\bar{E}_j = E_j / E_{j+1}\quad \textmd{and}\quad \varphi_{j}: \bar{E}_{j-1} \rightarrow \bar{E}_{j} \otimes L 
\]
is induced by $\Phi$.
\item The Higgs field $\Phi$ satisfies $\Phi^{2} = 0$ but does not
  vanish identically, and there is a filtration
 \[
  E = E_1 \supset E_2 \supset E_{3} \supset E_{4} = 0
 \]
with each step of co-dimension one and satisfying
$\Phi(E_j) \subset E_{j+2} \otimes L$ for $j=1,2$.
The topological invariants of $E_2$ and $E_3$ are constrained by the inequalities
\begin{equation}
  \label{eq:5}
  \mu(E) - \deg(L)/2 < \mu(E/E_2\oplus E_3) < \mu(E) + \deg(L)/2.
\end{equation}
Moreover, 
\begin{enumerate}[(i)]
 \item[(c.1.)]  if $\mu(E_1/E_2\oplus E_3) < \mu(E)$ then
   \begin{equation}
     \label{eq:3}
     (E^{\infty},\Phi^{\infty})
     = \displaystyle \lim_{z\rightarrow \infty}(E,z\cdot \Phi) = 
 \Big(E_1/E_2 \oplus E_2,
 \left(
    \begin{array}[c]{c c}
            0          &      0\\
       \varphi    &      0
    \end{array}
 \right)
 \Big)
  \end{equation}
is a Hodge bundle of type $(1,2)$ where $\varphi\colon E_1/E_2 \rightarrow E_2\otimes
L$ is induced by $\Phi$ and,
  \item[(c.2.)] if $\mu(E_1/E_2\oplus E_3) > \mu(E)$ then
    \begin{equation}
      \label{eq:6}
      (E^{\infty},\Phi^{\infty})
      = \displaystyle \lim_{z\rightarrow \infty}(E,z\cdot \Phi) = 
 \Big(E_1/E_3\oplus E_3,
 \left(
    \begin{array}[c]{c c}
            0          &      0\\
       \varphi    &      0
    \end{array}
 \right)
 \Big)
 \end{equation}
is a Hodge bundle of type $(2,1)$ where $\varphi\colon E_1/E_3 \rightarrow E_3\otimes
L$ is induced by $\Phi$.
 \end{enumerate}
\end{enumerate}
\end{Th}

\begin{proof}
If $\Phi$ vanishes identically it is clear that the statement of case (a) holds and,
when $\Phi$ is a regular nilpotent, the statement of case (b) follows from
Theorem~\ref{NilpotentConeRkR} with $r=3$.

It  remains to consider the case when $\Phi\neq 0$ and $\Phi^2 \equiv 0$. Then, we may consider:
\[
 E_2 = \widetilde{\ker(\Phi)} \subset E_1=E\quad\text{and}\quad\text E_3 =
 \widetilde{\im(\Phi)}\otimes L^{-1} \subset E_2,
\]
where the tildes indicate taking the saturation of a subsheaf. We note
that, necessarily from our assumptions on $\Phi$, that $\rk(E_2)=2$,
$\rk(E_3)=1$, and that we obtain a filtration with the properties
stated in case (c).

We proceed to prove the constraints \eqref{eq:5}. From stability of $(E,\Phi)$ we have the inequalities
\begin{align}
  \label{eq:7}
  \mu(E_3)&<\mu(E)  \iff 3\deg(E_3)<d,\\
  \label{eq:8}
  \mu(E_2)&<\mu(E)  \iff 3\deg(E_2)<2d,
\end{align}
since $E_2$ and $E_3$ are $\Phi$-invariant subbundles of
$E$. Moreover, $\Phi$ induces a non-zero map of line bundles
\begin{math}
  E/E_2\rightarrow E_3\otimes L
\end{math}
and hence
\begin{equation}
  \label{eq:9}
  \deg(E_3) + \deg(L) \geq d-\deg(E_2).
\end{equation}
Now, using \eqref{eq:9} and \eqref{eq:8} we obtain
\begin{align*}
  2\mu(E/E_2\oplus E_3) &= d-\deg(E_2)+\deg(E_3)\\
                        &\geq 2d-2\deg(E_2)-\deg(L) \\
  &>\frac{2}{3}d-\deg(L)
\end{align*}
which is the first of the inequalities \eqref{eq:5}. Similarly, from
using \eqref{eq:9} and \eqref{eq:7} we obtain
\begin{align*}
  2\mu(E/E_2\oplus E_3) &= d-\deg(E_2)+\deg(E_3)\\
                        &\leq 2\deg(E_3)+\deg(L) \\
  &<\frac{2}{3}d+\deg(L)
\end{align*}
which is the second of the inequalities \eqref{eq:5}.

It remains to identify the limit of $(E,z\cdot\Phi)$ as
$z\rightarrow\infty$. For this we take, as usual, a smooth splitting
\begin{displaymath}
  E \underset{C^\infty}{\cong} E_1/E_2\oplus E_2/E_3 \oplus E_3.
\end{displaymath}
With respect to this splitting we have, from the definitions of $E_2$
and $E_3$, that
\begin{displaymath}
  \Phi=\left(
    \begin{array}[c]{c c c}
      0       & 0 & 0\\
      0       & 0 & 0\\
      \varphi & 0 & 0
    \end{array}
 \right).
\end{displaymath}
With respect to each of the smooth splittings $E\cong E_1/E_2\oplus E_2$
and $E\cong E_1/E_3\oplus E_3$ we can take a family of smooth complex gauge
transformations $g(z)\in\mathcal{G}^\bC$ defined by
\[
 g(z) = 
 \left(
	    \begin{array}{c c}
	     1 & 0 \\
	     0 & z
	    \end{array}
 \right)
\]
(interpreting each entry as a block of the appropriate size). Exactly
the same argument as in the proof of Theorem~\ref{NilpotentConeRkR}
shows that we have the convergence in the configuration space, up to
gauge equivalence, stated in each of the sub-cases (c.1.) and
(c.2.). It remains to prove that the convergence also holds in the
moduli space, i.e., that the Hitchin pairs in \eqref{eq:3} and
\eqref{eq:6} are stable under the respective hypotheses on
$\mu(E/E_2\oplus E_3)$.

\paragraph{Case (c.1.)}
The proper non-trivial $\Phi^\infty$-invariant subbundles $F\subset E^\infty$ are of two kinds:
\begin{enumerate}[(1)]
\item $F\subseteq E_2\subseteq E^\infty$ any non-zero subbundle (which
  may equal $E_2$). In this case $F$ defines a
  $\Phi$-invariant subbundle of the stable Hitchin pair $(E,\Phi)$ and hence
  $\mu(F)<\mu(E)=\mu(E^\infty)$ as desired. 
\item $F= E_1/E_3\oplus E_3\subseteq E^\infty$. In this case
  $\mu(F) = \mu(E_1/E_3\oplus E_3) < \mu(E)=\mu(E^\infty)$ by
  hypothesis.
\end{enumerate}

\paragraph{Case (c.2.)}
Again, the proper non-trivial $\Phi^\infty$-invariant subbundles $F\subset E^\infty$ are of two kinds:
\begin{enumerate}[(1)]
\item $F=L \oplus E_3 \subseteq E^\infty$ for a proper subbundle
  $L\subseteq E_1/E_3$ (which may be zero). In this case we can lift
  $L$ to a subbundle $\tilde{L}\subset E$ and we note that
  $E_3\subseteq \tilde{L}$. Hence $V\subseteq E$ is $\Phi$-invariant
  and $\mu(L\oplus E_3)=\mu(V)<\mu(E)=\mu(E^\infty)$ as we wanted.
\item $F= E_2/E_3\subseteq E_1/E_3\subseteq E^\infty$. In this case
  $\mu(F) = \mu(E_2/E_3)$ and we have
  \begin{align*}
    \mu(E_1/E_2\oplus E_3) &=
                             \frac{1}{2}\big(3\mu(E_1)-2\mu(E_2)+\mu(E_3)\big)\\
    &=\frac{1}{2}\big(3\mu(E)-\mu(E_2/E_3)\big).
  \end{align*}
  Hence the hypothesis $\mu(E_1/E_2\oplus E_3)>\mu(E)$ is equivalent
  to $\mu(E_2/E_3)<\mu(E)$, as desired.
\end{enumerate}
\end{proof}

\section*{Acknowledgement}
\addcontentsline{toc}{section}{Acknowledgement}

The authors are members of the Vector Bundles and Algebraic Curves
(VBAC) research group.

\end{document}